\documentclass[11pt,a4paper]{amsart}
\input epsf

\usepackage{graphicx, amsmath, amssymb, amscd, amsthm, euscript, amsfonts, xcolor, latexsym, amsbsy}

\usepackage{tikz}
\usetikzlibrary{shapes.geometric, arrows.meta, positioning, calc}

\usepackage{hyperref}
\usepackage{float}
\usepackage{longtable}
\usepackage{subfigure}
\usepackage{graphics, geometry, caption}
\usepackage{makecell}

\newtheorem{thm}{Theorem}[section]
\newtheorem{cor}[thm]{Corollary}
\newtheorem{prop}[thm]{Proposition}
\newtheorem{lem}[thm]{Lemma}
\theoremstyle{definition}

\newcommand{\Rmnum}[1]{\uppercase\expandafter{\romannumeral #1}}

\title[Non-discrete complex hyperbolic $(p_1,p_2,p_3;p_4)$-triangle groups]{Non-discrete complex hyperbolic $(p_1,p_2,p_3;p_4)$-triangle groups}

\author[Wei Liao]{Wei Liao}
\address[Wei Liao]{School of Mathematics, Hunan University, Changsha 410082, P.R. China}
\email{liaowei.math@qq.com}

\author[John R. Parker]{John R. Parker}
\address[John R. Parker]{Department of Mathematical Sciences, Durham University, Durham DH1 3LE, United Kingdom}
\email{j.r.parker@durham.ac.uk}

\keywords{Complex hyperbolic triangle groups, Discreteness. 51M10, 22E40}

\begin{document}

\begin{abstract}
In this paper, we study complex hyperbolic triangle groups of type $(p_1,p_2,p_3;p_4)$ with $3\le p_1\le p_2\le p_3\le p_4\le \infty$. We show that if the element $R_1R_2R_3$ is elliptic, then the group is either non-discrete or it belongs to a short list of known discrete groups.
\end{abstract}

\maketitle

\section{Introduction}
A complex hyperbolic $(p_1,p_2,p_3)$-triangle group is a subgroup of $\mathrm{SU}(2,1)$ generated by complex involutions 
$R_1$, $R_2$, and $R_3$ with mirrors given by the complex geodesics $C_1$, $C_2$, and $C_3$, respectively. 
We also suppose the complex angle between $C_{j-1}$ and $C_{j+1}$ is $\pi/p_j$ (with indices taken modulo $3$), where 
$p_j\in {\mathbb N}\cup \{\infty\}$. When $p_j\in{\mathbb N}$, this implies $R_{j-1}R_{j+1}$ is elliptic of order $p_j$. 
When $p_j=\infty$, the complex geodesics $C_{j-1}$ and $C_{j+1}$ are understood to be asymptotic at the ideal boundary of 
$\mathbf H_{\mathbb C}^2$ and the angle between them is zero. This implies $R_{j-1}R_{j+1}$ is parabolic.

Unlike their real hyperbolic counterpart, if $\min\{p_1,p_2,p_3\}\ge 3$, a complex hyperbolic $(p_1,p_2,p_3)$-triangle group has a 
one-dimensional real parameter space, which can be parametrised by the trace of $R_1R_3R_2R_3$. 
(When $\min\{p_1,p_2,p_3\}=2$ then the group is rigid; we do not consider that case here.) We note that $R_1R_3R_2R_3$ is the 
product of complex involutions in the complex geodesics $C_1$ and $C_{323}=R_3(C_2)$. 
We are interested in the case where $C_1$ and $C_{323}$ intersect or are asymptotic and the angle between them is $\pi/p_4$ 
where $p_4\in{\mathbb N}\cup\{\infty\}$. This implies $R_1R_3R_2R_3$ is elliptic of order $p_4$ or is parabolic. 

A \emph{complex hyperbolic $(p_1,p_2,p_3;p_4)$-triangle group}, denoted by $\Gamma(p_1,p_2,p_3;p_4)$, is the group generated 
by the complex involutions $R_1$, $R_2$, $R_3$ described above. If this group is discree, then using Theorem~1.1 of our previous 
paper \cite{lp}, after changing the generators if necessary, we assume that $3\le p_1\le p_2\le p_3\le p_4\le \infty$.

The study of the discreteness of complex hyperbolic triangle groups in which $R_1R_2R_3$ is elliptic has a history spanning several 
decades, dating back to the case of complex hyperbolic ideal triangle groups studied by Goldman and Parker \cite{gp}. In that paper, 
they conjectured that when $R_1R_2R_3$ is elliptic, the complex hyperbolic ideal triangle group is non-discrete, and this conjecture was 
later proved by Schwartz in \cite{sch}. Subsequently, this result was generalised to complex hyperbolic $(p_1,p_2,\infty)$-triangle 
groups by Pratoussevitch \cite{pra2} and Sun \cite{sun}.

Based on the non-discreteness result in \cite{pra2} and complex hyperbolic versions of J\o rgensen's inequality, Kamiya, Parker, and 
Thompson \cite{kpt2} considered the groups of type $(p_1,p_1,\infty;p_4)$ (or equivalently $(p_1,p_1,p_4;\infty)$) and proved that t
hese groups are non-discrete for $p_1\ge 29$. This bound was later improved by Kamiya in \cite{kam1,kam2} to $p_1\ge 14$. A more 
general result was obtained by Kamiya, Parker, and Thompson~\cite{kpt1}, who used J\o rgensen's inequality to show that groups 
$\Gamma(p_1,p_2,p_3;p_4)$ are not discrete if $p_1\ge 32$. 

Another technical approach is to use the Conway--Jones argument \cite{cj} to prove non-discreteness. For example, Parker \cite{par3} 
applied this method to find all complex hyperbolic $(p_1,p_1,p_1)$-triangle groups for which the element $R_1R_2R_3$ has finite order 
and the group is discrete. Subsequently, Monaghan \cite{mon} and Parker and Sun \cite{ps} investigated, respectively, complex hyperbolic 
$(p_1,p_2,p_3;p_4)$-triangle groups of type $1/p_1=1/p_2+1/p_3$ with $p_1\ge 10$, and the case $p_1=p_2$ with $p_3=p_4$. 
Their results suggest that, under the condition $p_1\le p_2\le p_3$, these groups are non-discrete when $R_1R_2R_3$ is elliptic.

This result cannot be generalised to all complex hyperbolic $(p_1,p_2,p_3;p_4)$-triangle groups, since there are examples in which $R_1R_2R_3$ is elliptic, but the group is discrete. The known examples include the Deraux lattice $\Gamma(4,4,4;5)$ \cite{der}, the Thompson groups $\Gamma(3,3,4;7)$ and $\Gamma(3,3,5;5)$ \cite{tho} (see also \cite{kpt1}), and normal subgroups of the Livn\'{e} groups \cite{liv} (see also \cite{par2}):
\begin{gather*}
\Gamma(5,5,5;5),\quad \Gamma(7,7,7;7),\quad \Gamma(8,8,8;8), \\
\Gamma(9,9,9;9),\quad \Gamma(10,10,10;10),\quad \Gamma(12,12,12;12),\quad \Gamma(18,18,18;18).
\end{gather*}
(We note that $\Gamma(6,6,6;6)$ is also discrete but in this case $R_1R_2R_3$ is parabolic.)
These groups and their higher-order analogues constitute a significant source of both arithmetic and non-arithmetic lattices in $\mathrm{SU}(2,1)$; see \cite{dpp1,dpp2}.

In Section~2.6.1 of \cite{tho}, Thompson used a computer-aided approach to find all groups $\Gamma(p_1,p_2,p_3;p_4)$ satisfying $p_1<31$ and $p_2,p_3,p_4<2000$ for which the three elements $R_1R_2R_3$, $R_1R_2R_3R_2$, and $R_1R_2R_1R_3$ are all elliptic of finite order. As a result, apart from the groups $\Gamma(p_1,p_1,p_1;p_1)$, Thompson identified three groups: $\Gamma(3,3,4;7)$, $\Gamma(3,3,5;5)$, and $\Gamma(4,4,4;5)$. The third of these groups had
previously been constructed by Deraux \cite{der}.

Our main result is the following:

\begin{thm}\label{thm}
Let $\Gamma(p_1,p_2,p_3;p_4)$ be a complex hyperbolic triangle group with 
$3\le p_1\le p_2\le p_3\le p_4\le\infty$.
If the element $R_1R_2R_3$ is elliptic, then $\Gamma(p_1,p_2,p_3;p_4)$ is non-discrete, unless it is the Deraux lattice, a Thompson group, or a normal subgroup of a Livn\'{e} group.
\end{thm}

We remark that Theorem~\ref{thm} also applies to the case where the $R_j$ have finite order greater than two, as considered in~\cite{dpp1,dpp2}. This is because in that case we must solve the same equation~\eqref{main equality 1} as we use in our main results;
compare Parker \cite{par3} and Parker and Paupert \cite{pp}.

In \cite{sch2} Schwartz gave a number of conjectures about complex hyperbolic triangle
groups. As indicated above, given $(p_1,p_2,p_3)$ there is a one parameter family of $(p_1,p_2,p_3)$-triangle
groups. Schwartz defined two group elements $W_A=R_1R_3R_2R_3$ and $W_B=R_1R_2R_3$. A group is said to be
\begin{enumerate}
    \item[(a)] of Type~A if there is a parameter value where $W_A$ is parabolic and $W_B$ is loxodromic; and
    \item[(b)] of Type~B if there is a parameter value where $W_B$ is parabolic and $W_A$ is loxodromic.
\end{enumerate}
In Conjecture~5.3 of \cite{sch2}, Schwartz conjectured that for groups of Type~B, the group is not discrete for all parameter values for which $W_B$ is elliptic. In Proposition~7.5 of \cite{par1} Parker showed that $\Gamma(12,12,12;12)$ and $\Gamma(18,18,18;18)$ are both
counterexamples to this conjecture. A consequence of our main theorem is the following.

\begin{cor}\label{cor non-discrete group}
Let $\Gamma(p_1,p_2,p_3;p_4)$ be a Type~B complex hyperbolic triangle group with $3\le p_1\le p_2\le p_3\le p_4\le\infty$, in which both $W_A$ and $W_B$ are both elliptic. If $\Gamma(p_1,p_2,p_3;p_4)$ is discrete then it must be
either $\Gamma(12,12,12;12)$ or $\Gamma(18,18,18;18)$. In particular, these two groups
are the only counterexamples to Conjecture~5.3 of \cite{sch2} with $W_A$ elliptic.
\end{cor}

This paper is organised as follows. In Section~\ref{background}, we present the necessary background on complex hyperbolic space and complex hyperbolic $(p_1,p_2,p_3;p_4)$-triangle groups. In Section~\ref{sec outline}, we outline the proof of Theorem~\ref{thm}. In Section~\ref{section p1<=9}, we consider the case $3\le p_1\le 9$. Finally, in Sections~\ref{section p1=p2} and~\ref{section p1<p2}, we complete the proof of Theorem~\ref{thm} in the cases $p_1=p_2$ and $p_1<p_2$, respectively.

\medskip

\noindent{\bf Acknowledgments.}
This paper was written while the first author was visiting Durham University. We would like to thank the China Scholarship Council for their financial support (Grant No. 202506130039).

\medskip

\section{Background}\label{background}
We refer the reader to \cite{gol} for the basics of complex hyperbolic geometry.

\subsection{Complex hyperbolic space}
Let $\mathbb{C}^{2,1}$ denote the $3$-dimensional complex vector space equipped with a Hermitian form of signature $(2,1)$,
for example 
$$
\langle \mathbf{z},\mathbf{w}\rangle
=z_1\overline{w}_1+z_2\overline{w}_2-z_3\overline{w}_3.
$$
Define
$$
V_-=\left\{\mathbf{z}\in\mathbb{C}^{2,1}\setminus\{\mathbf{0}\}:
\langle\mathbf{z},\mathbf{z}\rangle<0\right\},
\qquad
V_0=\left\{\mathbf{z}\in\mathbb{C}^{2,1}\setminus\{\mathbf{0}\}:
\langle\mathbf{z},\mathbf{z}\rangle=0\right\}.
$$
Note that $z_3\neq 0$ in both $V_-$ and $V_0$. Let 
${\mathbb P}:{\mathbb C}^{2,1}\longrightarrow {\mathbb C}{\mathbb P}^2$ be the standard projection map, which in the chart where $z_3\neq 0$ is given by
$$
\begin{bmatrix}
z_1\\
z_2\\
z_3
\end{bmatrix}
\mapsto
\begin{bmatrix}
z_1/z_3\\
z_2/z_3
\end{bmatrix},
$$
Then complex hyperbolic 2-space and its boundary are defined to be:
$$
\mathbf{H}_{\mathbb{C}}^2=\mathbb{P}(V_-)
\quad\text{and}\quad
\partial\mathbf{H}_{\mathbb{C}}^2=\mathbb{P}(V_0).
$$
Complex hyperbolic space $\mathbf{H}_{\mathbb{C}}^2$ is equipped with the Bergman metric $\rho$, which satisfies
\[
\cosh^2\left(\frac{\rho(z,w)}{2}\right)
=
\frac{\langle \mathbf{z},\mathbf{w}\rangle
      \langle \mathbf{w},\mathbf{z}\rangle}
     {\langle \mathbf{z},\mathbf{z}\rangle
      \langle \mathbf{w},\mathbf{w}\rangle}.
\]
Its group of holomorphic isometries is
\[
\mathrm{PU}(2,1)
=
\mathrm{SU}(2,1)/\{I,\omega I,\omega^2 I\},
\]
where $\omega$ is a primitive cube root of unity.

Isometries are classified into three types according to their fixed-point sets. More precisely, an isometry is called \textit{parabolic} if it has exactly one fixed point on $\partial\mathbf{H}_{\mathbb{C}}^2$, \textit{loxodromic} if it has exactly two distinct fixed points on $\partial\mathbf{H}_{\mathbb{C}}^2$, and \textit{elliptic} if it has at least one fixed point in $\mathbf{H}_{\mathbb{C}}^2$.

Let ${\bf n}\in{\mathbb C}^{2,1}$ satisfy $\langle{\bf n},{\bf n}\rangle>0$. Then
the set of vectors ${\bf z}\in{\mathbb C}^{2,1}$ satisfying $\langle{\bf z},{\bf n}\rangle=0$
is a two (complex) dimensional subspace of ${\mathbb C}^{2,1}$ containing points of $V_-$.
The image under ${\mathbb P}$ of this set is a copy of ${\mathbb C}{\mathbb P}^1$ that intersects ${\bf H}^2_{\mathbb C}$ in a complex geodesic $C$. The vector ${\bf n}$ is called a polar vector for $C$. The map
$$
R:{\bf z}\longmapsto -\mathbf z
+\frac{2\langle{\bf z},{\bf n}\rangle}{\langle{\bf n},{\bf n}\rangle}\,{\bf n}
$$
is called a complex reflection is given by a matrix in ${\rm SU}(2,1)$ of order two. When 
this matrix acts projectively on ${\bf H}^2_{\mathbb C}$ then it fixes each point of $C$, 
which is called its \emph{mirror}.

\medskip

\subsection{Complex hyperbolic $(p_1, p_2, p_3; p_4)$-triangle groups}\label{sec basic on triangle groups} 

Consider three complex geodesics $C_1$, $C_2$, $C_3$ with polar vectors ${\bf n}_1$, 
${\bf n}_2$, ${\bf n}_3$ respectively. Let $R_1$, $R_2$, $R_3$ be the complex reflections 
with mirrors $C_1$, $C_2$, $C_3$ respectively. The complex angle $\phi_j$ between
$C_{j-1}$ and $C_{j+1}$, with indices mod $3$, is given by
$$
\cos^2(\phi_j)
=\frac{\langle{\bf n}_{j-1},{\bf n}_{j+1}\rangle \langle{\bf n}_{j+1},{\bf n}_{j-1}\rangle }
{\langle{\bf n}_{j-1},{\bf n}_{j-1}\rangle \langle{\bf n}_{j+1},{\bf n}_{j+1}\rangle }.
$$
Then 
$$
{\rm tr}(R_{j-1}R_{j+1})=4\cos^2(\phi_j)-1=1+2\cos(2\phi_j).
$$
We also suppose that the complex angle between $C_1$ and $C_{323}=R_3(C_2)$ is $\phi_4$. That is
$$
\cos^2(\phi_4)
=\frac{\langle{\bf n}_1,R_3({\bf n}_2)\rangle \langle R_3({\bf n}_2),{\bf n}_1\rangle }
{\langle{\bf n}_1,{\bf n}_1\rangle \langle R_3({\bf n}_2),R_3({\bf n}_2)\rangle }
$$
and
$$
{\rm tr}(R_1R_3R_2R_3)=4\cos^2(\phi_4)-1=1+2\cos(2\phi_4).
$$
We are particularly interested in the cases where $\phi_j=\pi/p_j$ with $p_j\in{\mathbb N}$ for
$j=1,\,\ldots,\,4$. In this case, the group generated by $R_1$, $R_2$ and $R_3$ is called
a complex hyperbolic $(p_1,p_2,p_3;p_4)$-triangle group and is denoted 
$\Gamma(p_1,p_2,p_3;p_4)$.

\begin{thm}[Theorem 1.1 of Liao--Parker \cite{lp}]\label{thm-lp}
    After applying a suitable automorphism to a discrete group $\Gamma(p_1,p_2,p_3;p_4)$ we may suppose 
    $p_1\le p_2\le p_3\le p_4\le \infty$.
\end{thm}

Following \cite{pra1}, the \textit{angular invariant} $\alpha$ of a $(p_1,p_2,p_3)$-triangle group is defined by
$$
\alpha
=
\arg\bigl(
\langle {\bf n}_1,{\bf n}_3\rangle
\langle {\bf n}_2,{\bf n}_1\rangle
\langle {\bf n}_3,{\bf n}_2\rangle
\bigr).
$$

\begin{prop}[Proposition~1 of \cite{pra1}]\label{main alpha}
A complex hyperbolic $(p_1,p_2,p_3)$-triangle group is uniquely determined, up to holomorphic isometry, by the triple 
$(p_1,p_2,p_3)$ and the angular invariant $\alpha$. For any $\alpha\in[0,2\pi]$, such a group exists if and only if
\begin{equation}\label{alpha}
\cos\alpha
<
\frac{r_1^2+r_2^2+r_3^2-1}{2r_1r_2r_3},
\end{equation}
where $r_k=\cos\phi_k$ and $\phi_k=\pi/p_k$.
\end{prop}

For a complex hyperbolic $(p_1,p_2,p_3;p_4)$-triangle group, suppose that the element $R_1R_2R_3$ is elliptic and has finite order. Then $R_1R_2R_3$ has three eigenvalues of unit modulus whose product is equal to $1$. We write these eigenvalues as $e^{i\theta_1}$, $e^{i\theta_2}$, and $e^{i\theta_3}$, where $\theta_j\in\mathbb{Q}\pi$ and
\begin{equation}\label{eq-theta-sum}
\theta_1+\theta_2+\theta_3=2k\pi
\end{equation}
for some $k\in\mathbb{Z}$. 
The traces of elements of a complex hyperbolic triangle group can be obtained from the results of \cite{pra1,par3} or by direct computation using the matrix representations of complex reflections given in \cite{kpt1}.
Using the traces of $R_1R_3R_2R_3$ and $R_1R_2R_3$, we obtain
\begin{equation}\label{1323}
\mathrm{tr}(R_1R_3R_2R_3)
= 4r_4^2-1
=16r_1^2r_2^2+4r_3^2-1-16r_1r_2r_3\cos\alpha,
\end{equation}
and
\begin{equation}\label{123}
\mathrm{tr}(R_1R_2R_3)
=
e^{i\theta_1}+e^{i\theta_2}+e^{i\theta_3}
=
8r_1r_2r_3e^{i\alpha}
-4(r_1^2+r_2^2+r_3^2)+3.
\end{equation}
In particular, this shows we may take ${\rm tr}(R_1R_3R_2R_3)$ instead of $\alpha$ as the
parameter for the space of $(p_1,p_2,p_3)$-triangle groups. 
 It follows from \eqref{123} that inequality~\eqref{alpha} is equivalent to
\begin{equation}\label{signature (2,1)}
 \mathrm{Re}\bigl(\mathrm{tr}(R_1R_2R_3)\bigr)
=\cos(\theta_1)+\cos(\theta_2)+\cos(\theta_3)<-1.   
\end{equation}

\begin{lem}\label{lem-not-Zpi}
   Suppose that $\theta_1$, $\theta_2$, $\theta_3$ satisfy 
   \eqref{eq-theta-sum} and \eqref{signature (2,1)}.
   Then $\theta_j\not\in{\mathbb Z}\pi$.
\end{lem}

\begin{proof}
Suppose for a contradiction that $\theta_1\in{\mathbb Z}\pi$. First, suppose that
$\theta_1=(2\ell+1)\pi$ for $\ell\in{\mathbb Z}$. Then $\cos(\theta_1)=-1$ and 
$$
\theta_2+\theta_3=2k\pi-\theta_1=(2k-2\ell-1)\pi.
$$
Thus $\cos(\theta_3)=-\cos(\theta_2)$ and
$$
\cos(\theta_1)+\cos(\theta_2)+\cos(\theta_3)=-1+\cos(\theta_2)-\cos(\theta_2)=-1.
$$
Next suppose that $\theta_1=2\ell\pi$. Then $\cos(\theta_1)=1$ and 
$\theta_2+\theta_3=2(k-\ell)\pi$ and so $\cos(\theta_3)=\cos(\theta_2)$. Hence
$$
\cos(\theta_1)+\cos(\theta_2)+\cos(\theta_3)=1+2\cos(\theta_2)\ge -1.
$$
In both cases we have a contradiction.    
\end{proof}

\medskip

We rewrite \eqref{1323} as
\begin{equation}\label{alpha and n}
\cos\alpha
=
\frac{4r_1^2r_2^2+r_3^2-r_4^2}
     {4r_1r_2r_3}.
\end{equation}
Substituting \eqref{alpha and n} into \eqref{123} and then taking the real part yields
\begin{equation}\label{equality}
8r_1^2r_2^2-4r_1^2-4r_2^2-2r_3^2-2r_4^2
-\cos\theta_1-\cos\theta_2-\cos\theta_3+3
=
0.
\end{equation}
Following Monaghan \cite{mon}, equation \eqref{equality} can be simplified to
\begin{equation}\label{main equality 1}
\cos(2\phi_1+2\phi_2)
+\cos(2\phi_1-2\phi_2)
-\cos(2\phi_3)
-\cos(2\phi_4)
-\cos\theta_1-\cos\theta_2-\cos\theta_3
=1.
\end{equation}

We then obtain the following result.

\begin{lem}\label{nec}
Let $\Gamma(p_1,p_2,p_3;p_4)$ be a discrete complex hyperbolic triangle group, and suppose that $R_1R_2R_3$ is elliptic with eigenvalues $e^{i\theta_j}$ for $j=1,2,3$. If $R_1R_2R_3$ has finite order, then $\theta_j\in\mathbb{Q}\pi\setminus\mathbb{Z}\pi$ for $j=1,2,3$, and these angles satisfy the equations~\eqref{eq-theta-sum}~\eqref{signature (2,1)} and~\eqref{main equality 1}.
\end{lem}

For a fixed complex hyperbolic $(p_1,p_2,p_3;p_4)$-triangle group whose element $R_1R_2R_3$ is elliptic, one strategy for proving that $R_1R_2R_3$ has infinite order is to show that equation~\eqref{main equality 1} admits no solutions for which $\theta_j$ is a rational multiple of $\pi$.
For this purpose, we repeatedly apply the following result from Monaghan \cite{mon}, which generalises the results of Conway and Jones \cite{cj}, see also Parker \cite{par1}. Monaghan
lists all rational sums of at most seven cosines. We only need to use the details of the sums of length at most six, together with the general form of sums of length seven. To save space, we do not list all the sums of length seven. 

\begin{cor}[Corollary of Theorem~2.4.3.1 in \cite{mon}]\label{de1}
Suppose that we have exactly seven distinct rational multiples of $\pi$ in the interval $(0,\pi/2)$ such that a rational linear combination of their cosines is equal to a rational number $r$, and that no proper subset has this property. Then, after
multiplying by a constant if necessary, the coefficient of each cosine term is $\pm1$, and the resulting rational value $r$ is either $0$ or $1/2$.
\end{cor}

\begin{thm}[Theorem~2.4.3.1 of \cite{mon}]\label{de2}
    Suppose we have at most six distinct rational multiples of $\pi$, for which some rational linear combination of their cosines is rational but no proper subset has this property. If $\varphi\in (0, \pi)$ and all other angles are normalised to lie in $(0, \pi/2)$, then the appropriate linear combination is proportional to one of the following:
\begin{enumerate}
\item[(i)] $\displaystyle{0 = \sum_{k=0}^{2} \cos\left(\varphi+\frac{2k\pi}{3} \right ), \quad \varphi\in (0, \pi), \quad \varphi\ne \frac{m\pi}{6}}$, 
\item[(ii)] $\displaystyle{0 = \sum_{k=0}^{4} \cos\left(\varphi+\frac{2k\pi}{5} \right ), \quad \varphi\in (0, \pi), \quad \varphi\ne \frac{m\pi}{10}, \frac{l\pi}{15}}$, 
\item[(iii)] $\displaystyle{0 = \sum_{k=1}^{2} \cos\left(\varphi+\frac{2k\pi}{3} \right )-\sum_{k=1}^{4} \cos\left(\varphi+\frac{2k\pi}{5} \right ), \quad \varphi\in (0, \pi), \quad \varphi\ne \frac{m\pi}{6}, \frac{l\pi}{10}, \frac{s\pi}{15} }$, 
\item[(iv)] $\displaystyle{\frac{1}{2}  =  \cos \left(\frac{\pi}{3}\right) }$, 
\item[(v)] $\displaystyle{\frac{1}{2}  =  \cos \left(\frac{\pi}{5}\right)- \cos \left(\frac{2\pi}{5}\right) }$, 
\item[(vi)] $\displaystyle{\frac{1}{2}  =  \cos \left(\frac{\pi}{5}\right)- \cos \left(\frac{\pi}{15}\right)+ \cos \left(\frac{4\pi}{15}\right) }$, 
\item[(vii)] $\displaystyle{\frac{1}{2}  =  -\cos \left(\frac{2\pi}{5}\right)+ \cos \left(\frac{2\pi}{15}\right)- \cos \left(\frac{7\pi}{15}\right) }$, 
\item[(viii)] $\displaystyle{\frac{1}{2}  =  -\cos \left(\frac{\pi}{15}\right)+ \cos \left(\frac{2\pi}{15}\right)+ \cos \left(\frac{4\pi}{15}\right)-\cos \left(\frac{7\pi}{15}\right) }$,  
\item[(ix)] $\displaystyle{\frac{1}{2}  =  \cos \left(\frac{\pi}{7}\right)- \cos \left(\frac{2\pi}{7}\right)+ \cos \left(\frac{3\pi}{7}\right) }$,  
\item[(x)] $\displaystyle{\frac{1}{2}  =  \cos \left(\frac{\pi}{7}\right)- \cos \left(\frac{2\pi}{7}\right)+ \cos \left(\frac{2\pi}{21}\right)-\cos \left(\frac{5\pi}{21}\right)}$, 
\item[(xi)] $\displaystyle{\frac{1}{2}  =  \cos \left(\frac{\pi}{7}\right)+ \cos \left(\frac{3\pi}{7}\right)- \cos \left(\frac{\pi}{21}\right)+\cos \left(\frac{8\pi}{21}\right) }$,   
\item[(xii)] $\displaystyle{\frac{1}{2}  =  -\cos \left(\frac{2\pi}{7}\right)+ \cos \left(\frac{3\pi}{7}\right)+ \cos \left(\frac{4\pi}{21}\right)+\cos \left(\frac{10\pi}{21}\right) }$,  
\item[(xiii)] $\displaystyle{\frac{1}{2} =  \cos \left(\frac{\pi}{7}\right)- \cos \left(\frac{\pi}{21}\right)+ \cos \left(\frac{2\pi}{21}\right)-\cos \left(\frac{5\pi}{21}\right)+\cos \left(\frac{8\pi}{21}\right) }$,  
\item[(xiv)] $\displaystyle{\frac{1}{2}  =  -\cos \left(\frac{2\pi}{7}\right)+\cos \left(\frac{2\pi}{21}\right)+ \cos \left(\frac{4\pi}{21}\right)-\cos \left(\frac{5\pi}{21}\right)+\cos \left(\frac{10\pi}{21}\right) }$,   
\item[(xv)] $\displaystyle{\frac{1}{2}  =  \cos \left(\frac{3\pi}{7}\right)- \cos \left(\frac{\pi}{21}\right)+ \cos \left(\frac{4\pi}{21}\right)+\cos \left(\frac{8\pi}{21}\right)+\cos \left(\frac{10\pi}{21}\right) }$,   
\item[(xvi)] $\displaystyle{\frac{1}{2}  =  -\cos \left(\frac{\pi}{21}\right)+ \cos \left(\frac{2\pi}{21}\right)+ \cos \left(\frac{4\pi}{21}\right)-\cos \left(\frac{5\pi}{21}\right)+\cos \left(\frac{8\pi}{21}\right)+\cos \left(\frac{10\pi}{21}\right) }$,  
\item[(xvii)] $\displaystyle{\frac{1}{2}  =  \cos \left(\frac{\pi}{11}\right)- \cos \left(\frac{2\pi}{11}\right)+ \cos \left(\frac{3\pi}{11}\right)-\cos \left(\frac{4\pi}{11}\right)+\cos \left(\frac{5\pi}{11}\right) }$, 
\item[(xviii)] $\displaystyle{\frac{1}{2}  =  \cos \left(\frac{\pi}{11}\right)+ \cos \left(\frac{3\pi}{11}\right)- \cos \left(\frac{4\pi}{11}\right)+\cos \left(\frac{5\pi}{11}\right)-\cos \left(\frac{5\pi}{33}\right)+\cos \left(\frac{16\pi}{33}\right) }$,  
\item[(xix)] $\displaystyle{\frac{1}{2}  =  \cos \left(\frac{\pi}{11}\right)- \cos \left(\frac{2\pi}{11}\right)+ \cos \left(\frac{3\pi}{11}\right)+\cos \left(\frac{5\pi}{11}\right)-\cos \left(\frac{\pi}{33}\right)+\cos \left(\frac{10\pi}{33}\right) }$,   
\item[(xx)] $\displaystyle{\frac{1}{2}  =  \cos \left(\frac{\pi}{11}\right)- \cos \left(\frac{2\pi}{11}\right)+ \cos \left(\frac{3\pi}{11}\right)-\cos \left(\frac{4\pi}{11}\right)+\cos \left(\frac{4\pi}{33}\right)-\cos \left(\frac{7\pi}{33}\right) }$,   
\item[(xxi)] $\displaystyle{\frac{1}{2}  =  \cos \left(\frac{\pi}{11}\right)- \cos \left(\frac{3\pi}{11}\right)- \cos \left(\frac{4\pi}{11}\right)+\cos \left(\frac{5\pi}{11}\right)+\cos \left(\frac{2\pi}{33}\right)-\cos \left(\frac{13\pi}{33}\right) }$,   
\item[(xxii)] $\displaystyle{\frac{1}{2}  =  -\cos \left(\frac{2\pi}{11}\right)+ \cos \left(\frac{3\pi}{11}\right)- \cos \left(\frac{4\pi}{11}\right)+\cos \left(\frac{5\pi}{11}\right)+\cos \left(\frac{8\pi}{33}\right)+\cos \left(\frac{14\pi}{33}\right) }$,  
\item[(xxiii)] $\displaystyle{\frac{1}{2}  =  \cos \left(\frac{\pi}{13}\right)- \cos \left(\frac{2\pi}{13}\right)+ \cos \left(\frac{3\pi}{13}\right)-\cos \left(\frac{4\pi}{13}\right)+\cos \left(\frac{5\pi}{13}\right)-\cos \left(\frac{6\pi}{13}\right) }$,  
\item[(xxiv)] $\displaystyle{\frac{1}{2}  =  \cos \left(\frac{\pi}{7}\right)+ \cos \left(\frac{3\pi}{7}\right)- \cos \left(\frac{3\pi}{35}\right)+\cos \left(\frac{4\pi}{35}\right)-\cos \left(\frac{11\pi}{35}\right)-\cos \left(\frac{17\pi}{35}\right) }$,  
\item[(xxv)] $\displaystyle{\frac{1}{2}  =  \cos \left(\frac{\pi}{7}\right)- \cos \left(\frac{2\pi}{7}\right)- \cos \left(\frac{\pi}{35}\right)+\cos \left(\frac{6\pi}{35}\right)+\cos \left(\frac{8\pi}{35}\right)-\cos \left(\frac{13\pi}{35}\right) }$,  
\item[(xxvi)] $\displaystyle{\frac{1}{2}  =  -\cos \left(\frac{2\pi}{7}\right)+ \cos \left(\frac{3\pi}{7}\right)+ \cos \left(\frac{2\pi}{35}\right)-\cos \left(\frac{9\pi}{35}\right)+\cos \left(\frac{12\pi}{35}\right)+\cos \left(\frac{16\pi}{35}\right) }$. 
 \end{enumerate}

\end{thm}
\medskip
We call a sum \emph{minimal} if it has no proper subsum equal to $0$ or $\pm 1/2$.

\section{Outline of the proof of Theorem \ref{thm}}\label{sec outline}

Our strategy to establish non-discreteness is to show that if $R_1R_2R_3$ is elliptic, then it must have infinite order; which implies 
the group cannot be discrete. More precisely, suppose that $R_1R_2R_3$ is elliptic. If it has finite order then its eigenvalues 
have the form $e^{i\theta_j}$ where $\theta_1$, $\theta_2$, $\theta_3$ are rational multiples of $\pi$ satisfying 
\eqref{eq-theta-sum} and~\eqref{signature (2,1)}. Using Lemma~\ref{nec}, to prove that $R_1R_2R_3$ has infinite order, 
we first determine whether equation~\eqref{main equality 1} admits solutions  by repeatedly applying Theorem~\ref{de2}. 
If equation~\eqref{main equality 1} admits no such solutions, then $R_1R_2R_3$ necessarily has infinite order. 

To solve equation~\eqref{main equality 1}, we proceed in the following three steps.
\medskip

\emph{Step 1: The case $3\le p_1\le 9$.}
In Section~6 of \cite{lp}, we gave a complete classification of the isometry type of $R_1R_2R_3$ for $3\le p_1\le 9$. Consequently, 
for all groups $\Gamma(p_1,p_2,p_3;p_4)$ in this range, the element $R_1R_2R_3$ is loxodromic, with only finitely many exceptions. 
Among these exceptions are the Thompson groups $\Gamma(3,3,4;7)$, $\Gamma(3,3,5;5)$, 
the Deraux group $\Gamma(4,4,4,5)$ and $\Gamma(p_1,p_1,p_1;p_1)$ which are 
normal subgroups of Livn\'e lattices. The remaining exceptions occur only when $p_1\in\{9,8\}$. The group is not discrete when either
$p_1=8$ or $p_1=p_2=9$, and are non-discrete; see Lemma~\ref{lem p<=9}. 
Therefore, in proving Theorem~\ref{thm}, it suffices to consider $p_1\ge 10$ in the case $p_1=p_2$ and $p_1\ge 9$ in the case 
$p_1<p_2$, significantly reducing the required computation. 
\medskip

\emph{Step 2. The case $10\le p_{1} = p_{2}$.} This part is discussed in Section~\ref{section p1=p2}. In this case, 
equation~\eqref{main equality 1} becomes
    $$
    \cos(\pi-4\phi_{1}) + \cos(2\phi_{3}) + \cos(2\phi_{4}) + \cos\theta_1 + \cos\theta_2 + \cos\theta_3 = 0.
    $$
    According to Theorem \ref{de2}, two possibilities may occur. First, this equation may have no proper vanishing subsums, in which case it must have the form of the sums (of length $6$) listed in Theorem \ref{de2}. The corresponding solutions are then obtained, as discussed in Lemma \ref{lem4}. The second possibility is that there exist proper vanishing subsums. Depending on the lengths of these subsums, there are three cases:
    \begin{enumerate}
        \item[(a)] There exists a subsum of length $1$ equal to $0$ (Lemma \ref{lem cos x=0}).
        \item[(b)] There exists a subsum of length $2$ equal to $0$ (Lemma \ref{lem2}).
        \item[(c)] There exists a subsum of length $3$ equal to $0$ (Lemma \ref{lem3}).
    \end{enumerate}
    After checking all these subcases by repeatedly applying Theorem~\ref{de2} and verifying that 
    equations~\eqref{eq-theta-sum} and~\eqref{main equality 1} admit no solutions, we conclude that, in the case 
    $p_{1}=p_{2}$, if $R_1R_2R_3$ is elliptic, then it must have infinite order.
\medskip

    \emph{Step 3. The case $9\le p_{1} < p_{2}$.} This case is treated in Section~\ref{section p1<p2}, and the equation~\eqref{main equality 1} under consideration is
$$
\cos(2\phi_1+2\phi_2)+\cos(2\phi_1-2\phi_2)-\cos(2\phi_3)-\cos(2\phi_4)-\cos\theta_1-\cos\theta_2-\cos\theta_3=1.
$$
According to Corollary~\ref{de1}, this equation must decompose into several minimal subsums, two of which have value $1/2$, while any additional subsums, if present, have value $0$. By analysing the subsums of value $1/2$ listed in Theorem~\ref{de2} that contain a cosine term of the form $-\cos(2\pi/p)$, we obtain the unique solution, which corresponds to the group $\Gamma(10,30,30;30)$. Finally, we prove that $\Gamma(10,30,30;30)$ is non-discrete by exhibiting an elliptic element of infinite order.

Therefore, we prove Theorem~\ref{thm}.

\medskip

\section{The case $3\le p_1\le 9$}\label{section p1<=9}
In this section, we prove the following lemma.

\begin{lem}\label{lem p<=9}
Suppose that $\Gamma(p_1,p_2,p_3;p_4)$ is a complex hyperbolic triangle group such that $R_1R_2R_3$ is elliptic. If $p_1=8$ or $p_1=p_2=9$, then $\Gamma(p_1,p_2,p_3;p_4)$ is either a normal subgroup of the Livn\'{e} group or non-discrete.
\end{lem}

\begin{proof}
By Lemma~6.3 of \cite{lp}, any such group that is not a normal subgroup of a Livn\'{e} group must be one of the following:
\[
\Gamma(8,9,9;9),\qquad
\Gamma(8,8,8;9),\qquad
\Gamma(9,9,10;10),\qquad
\Gamma(9,9,9;p_4),
\]
where $p_4\in\{10,11,12\}$. The non-discreteness of all these groups except $\Gamma(8,9,9;9)$ follows from \cite{par3,ps}.

For the group $\Gamma(8,9,9;9)$, a direct computation gives
\begin{eqnarray*}
{\rm tr}(R_1R_2R_1R_3) 
& = & \bigl(1+2\cos(2\pi/9)\bigr)\bigl(1+2\cos(2\pi/9)-2\cos(\pi/4)\bigr) \\
& = & 2.8305\ldots\in(-1,3),
\end{eqnarray*}
so this element is elliptic. We then consider the equation
$$
{\rm tr}(R_1R_2R_1R_3)=1+2\cos(\phi),
$$
which is equivalent to
$$
1
=
\cos\left(\frac{\pi}{4}\right)
+\cos\left(\frac{\pi}{36}\right)
+\cos\left(\frac{17\pi}{36}\right)
-2\cos\left(\frac{2\pi}{9}\right)
-\cos\left(\frac{4\pi}{9}\right)
+\cos\phi.
$$

If $\phi=k\pi/2$ for some $k\in\mathbb{Z}$, then $\cos\phi\in\{0,\pm1\}$. It is clear that, in this case, the equation above does not hold. 
If $\phi\in\mathbb{Q}\pi$ but $\phi$ is not covered by the previous case, then Theorem~\ref{de2} implies that the right-hand side must 
decompose into two sums, each equal to $1/2$ and each containing the term $-\cos(2\pi/9)$. Checking all possible decompositions of 
the terms not involving $\cos\phi$, we find that none of them has value $1/2$. Therefore, the equation admits no solution with 
$\phi/\pi\in\mathbb{Q}$. It follows that $R_1R_2R_1R_3$ has infinite order, and hence $\Gamma(8,9,9;9)$ is non-discrete.
\end{proof}

\medskip

Therefore, by Theorem~1.4 of \cite{lp}, it suffices, in proving Theorem~\ref{thm}, to consider only the cases described in Section~\ref{sec outline}.
\medskip

\section{The case $p_1 = p_2$}\label{section p1=p2}

We suppose that $10\le p_1=p_2\le p_3\le p_4\le \infty$. Then equation~\eqref{main equality 1} becomes
\begin{equation}\label{eq-5.1}
        \cos(\pi-4\phi_{1}) + \cos(2\phi_{3}) + \cos(2\phi_{4}) + \cos\theta_1 + \cos\theta_2 + \cos\theta_3 = 0,
\end{equation}
where $\phi_j=\pi/p_j$. As explained in Sections~\ref{sec basic on triangle groups} and~\ref{sec outline}, to prove that if the element 
$R_1R_2R_3$ in the group $\Gamma(p_1, p_2, p_3; p_4)$ is elliptic with infinite order, it suffices to show that \eqref{eq-5.1}
has no solutions for $p_1, p_3, p_4\in \mathbb Z^+$ and $\theta_j\in \mathbb Q\pi$, where $j=1,\,2,\,3$, satisfting
\eqref{eq-theta-sum} and \eqref{signature (2,1)}.
Additionally, for the sake of convenience, we define
\begin{equation}\label{eq-T}
    T_1=\left\{\pi-4\phi_1\right\}, \qquad
    T_2=\left\{2\phi_3, 2\phi_4\right\}, \qquad
    T_3=\left\{\theta_1, \theta_2, \theta_3\right\}.
\end{equation}

According to Lemma~\ref{lem p<=9}, it suffices to restrict our attention to the case \( p_1 \ge 10 \).
The cases $p_1=p_3$, $p_3=p_4$, and either $p_3=\infty$ or $p_4=\infty$ have been treated in Parker \cite{par3}, 
Parker and Sun \cite{ps}, and Pratoussevitch \cite{pra2}, respectively. Furthermore, the case 
$p_3=2p_1$, and so by symmetry $p_4=2p_1$, was treated in Monaghan \cite{mon}. Thus, we may further assume that 
$10\le p_1<p_3<p_4<\infty$ and $p_3,\,p_4\neq 2p_1$. 

The main result of this section is stated as follows.

\begin{prop}\label{main prop 1}
    Suppose that a complex hyperbolic $(p_1, p_1, p_3; p_4)$-triangle group satisfies $10\le p_1 < p_3< p_4<\infty$ and 
    $p_3,\,p_4\neq 2p_1$. If the element $R_1R_2R_3$ is elliptic then the group is not discrete.
\end{prop}

We prove this result in Section~\ref{sec 5.2} and first present several technical lemmas in Section~\ref{sec 5.1}.

\medskip

\subsection{Technical lemmas}\label{sec 5.1}

In \cite{lp} we find all $\Gamma(p_1,p_2,p_3;p_4)$ in which $R_1R_2R_3$ is elliptic. 
In particular, the following result follows directly from Table~1 of \cite{lp}.

\begin{prop}\label{prop-lp}
Suppose that $10\le p_2\le 13$ and $p_1=p_2\le p_3\le p_4$. If the element $R_1R_2R_3$ of the group $\Gamma(p_1,p_1,p_3;p_4)$ is elliptic then:
\begin{enumerate}
    \item[(a)] If $p_1=p_2=10$ then $p_3\le 13$ and $p_4\le 20$;
    \item[(b)] If $p_1=p_2=11$ then $p_3\le 17$ and $p_4\le 41$; 
    \item[(c)] If $p_1=p_2=12$ then $p_3\le 24$ and $p_4\le 65$;
    \item[(d)] If $p_1=p_2=13$ then $p_3\le 56$ and $p_4\le 184$.
\end{enumerate}    
\end{prop}

We will frequently use equation~\eqref{signature (2,1)}, which states
$$
\mathrm{Re}\bigl(\mathrm{tr}(R_1R_2R_3)\bigr)=\cos\theta_1+\cos\theta_2+\cos\theta_3<-1,
$$
In particular, using equation~\eqref{eq-5.1}, we require that
\begin{equation}\label{eq-necessary}
\cos(\pi-4\phi_{1}) + \cos(2\phi_{3}) + \cos(2\phi_{4}) = -\cos\theta_1 - \cos\theta_2 - \cos\theta_3 >1.   
\end{equation}

From this, the following result is obtained immediately.

\begin{lem}\label{lem-subsum-length-2}
Let $T_1$, $T_2$, $T_3$ be as in equation \eqref{eq-T}. The following statements hold:
\begin{enumerate}
    \item[(a)] If $a,\,b\in (T_1\cup T_2)$ or $a,\,b\in T_3$, then the subsum $\cos a+\cos b$ of \eqref{eq-5.1} does not vanish.
    \item[(b)] If $a\in T_1$ and $b, c\in T_j$ for $j=2, 3$, then the subsum $\cos a+\cos b+\cos c$ of \eqref{eq-5.1} does not vanish.
\end{enumerate}
\end{lem}

\begin{proof}
In both cases described in part~(a), we see that equation~\eqref{eq-necessary} does not hold.
For part (b) it is clear that if $b,c\in T_2$ then
$$
0=\cos(a)+\cos(b)+\cos(c)=\cos(\pi-4\phi_1)+\cos(2\phi_3)+\cos(2\phi_4),
$$
which again contradicts \eqref{eq-necessary}.
Finally, suppose $b,\,c\in T_3$. Since $p_1\ge 10$ we see that 
$\pi-4\phi_1=\pi-4\pi/p_1\ge 3\pi/5$ and so $\cos(a)=\cos(\pi-4\phi_1)<0$. 
For $d\in T_3-\{b,c\}$, if $\cos(a)+\cos(b)+\cos(c)=0$ then we have
$$
1<-\cos(b)-\cos(c)-\cos(d)=\cos(a)-\cos(d)<-\cos(d).
$$
This is again a contradiction.
\end{proof}

\medskip

We now find possible values of $(p_1,p_j)$, for $j=3,4$, in which 
$\cos(\pi-4\phi_1)+\cos(2\phi_j)$
is a subsum of $S_1$ or $S_2$, which are the sums from Theorem~\ref{de2} (i) and (ii):
\begin{eqnarray}
S_1 & = &  \sum_{k=0}^{2} \cos\left( \varphi + \frac{2k\pi}{3} \right), \qquad \varphi \in (0, \pi), \quad \varphi \ne \frac{m\pi}{6}, \label{S1} \\
S_2 & = &  \sum_{k=0}^{4} \cos\left(\varphi + \frac{2k\pi}{5} \right), \qquad  \varphi \in (0, \pi), \quad \varphi \ne \frac{l\pi}{10}.\label{S2}
\end{eqnarray}
We can exclude the case where either $p_3$ or $p_4$ equals $2p_1$ using Theorem~2.2.0.17 of Monaghan \cite{mon} .

\begin{lem}\label{tech lemma 1}
    Suppose that $10\le p_1<p_3<p_4<\infty$ and $p_3,p_4\neq 2p_1$. Let $j\in\{3,\,4\}$. Then
    \begin{enumerate}
        \item[(a)] $2\pi/5<(\pi-4\pi/p_1)\pm 2\pi/p_j<\pi$ and $0<2\pi/p_3\pm 2\pi/p_4<2\pi/5$. 
        \item[(b)] If $\cos(\pi-4\pi/p_1)+\cos(2\pi/p_j)$
        is a subsum of $\pm S_1$ then $(p_1,p_j)$ is one of the pairs
        $$
        (10, 30),\ (11, 66),\ (13, 78),\ (14, 42)
        ,\ (16, 24).
        $$
\item[(c)] If $\cos(\pi-4\pi/p_1)+\cos(2\pi/p_j)$ 
is a subsum of $\pm S_2$ then $(p_1,p_j)$ is one of the pairs 
$$
\begin{array}{lllll}
(12,15),\,&  
(16,40),\,& (18,90),\, &(19,190),\\ 
(21,210),\,& (22,110),\,& 
(24,60),\,&(28,35).
\end{array}
$$
\item[(d)] $\cos(2\pi/p_3)+\cos(2\pi/p_4)$ is not a proper subsum of either $\pm S_1$ or 
$\pm S_2$.
    \end{enumerate}
\end{lem}

\begin{proof}
We set $j=4$. 
\begin{enumerate}
    \item[(a)] 
We have 
$$
0<2\phi_4=\frac{2\pi}{p_4}<2\phi_3=\frac{2\pi}{p_3}<2\phi_1=\frac{2\pi}{p_1}
\le \frac{\pi}{5}.
$$
Therefore
$$
(\pi-4\phi_1)\pm 2\phi_4\le (\pi-4\phi_1)+2\phi_4<\pi-2\phi_1<\pi.
$$
Similarly
$$
(\pi-4\phi_1)\pm 2\phi_4\ge (\pi-4\phi_1)-2\phi_4 > \pi-6\phi_1\ge \pi-\frac{3\pi}{5}=\frac{2\pi}{5}.
$$
Finally,
$$
0<2\phi_3-2\phi_4<2\phi_3+2\phi_4<4\phi_1\le \frac{2\pi}{5}.
$$

\medskip
\item[(b)] 
If $\cos(\pi-4\phi_1)+\cos(2\phi_4)=\cos(\pi-4\pi/p_1)+\cos(2\pi/p_4)$
is a proper subsum of $S_1$ then $(\pi-4\pi/p_1)$ and $\pm 2\pi/p_4$ differ by an integer multiple of $2\pi/3$. From (a) this multiple of $2\pi/3$ must lie strictly between $2\pi/5$ and $\pi$. Hence, it must be equal to $2\pi/3$. Therefore we must have
\begin{eqnarray*}
    \cos(2\pi/p_4) & = & \cos\bigl((\pi-4\pi/p_1)-2\pi/3\bigr) \\
    & = & \cos\bigl((p_1-12)\pi/3p_1\bigr).
\end{eqnarray*}
This has solutions when
$\bigl|6p_1/(p_1-12)\bigr|$ is an integer $p_4$ bigger than $p_1$. Note that for $p_1>12$ this is equivalent to $6>p_1-12$, which is $p_1<18$. We check each value of $p_1$ between $10$ and $17$. The only solutions with $p_j\neq 2p_1$ are the ones given above.
\medskip
\item[(c)] 
Similarly, if $\cos(\pi-4\phi_1)+\cos(2\phi_4)=\cos(\pi-4\pi/p_1)+\cos(2\pi/p_4)$
is a subsum of $S_2$ then $(\pi-4\pi/p_1)$ and $\pm 2\pi/p_4$ differ by an integer multiple of $2\pi/5$ lying strictly between $2\pi/5$ and $\pi$. Hence, this multiple must be equal to $4\pi/5$. Therefore, we must have
\begin{eqnarray*}
\cos(2\pi/p_4) & = & \cos\bigl((\pi-4\pi/p_1)-4\pi/5\bigr) \\
& = & \cos\bigl((p_1-20)\pi/5p_1\bigr).    
\end{eqnarray*}
This has solutions when $\bigl|10p_1/(p_1-20)\bigr|$ is an integer $p_4$ greater than $p_1$. Note that this is equivalent to 
$10>p_1-20>-10$, which is $10<p_1<30$. 
We check all $p_1$ between $11$ and $29$. The only solutions with $p_j\neq 2p_1$ 
are those given above.
\medskip
\item[(d)] If $\cos(2\phi_3)+\cos(2\phi_4)$ is a proper subsum of $S_1$ or $S_2$, respectively,  then
    $2\phi_3\pm2\phi_4=2\pi/p_3\pm 2\pi/p_4$ is a non-zero integer multiple of $2\pi/3$ or $2\pi/5$, respectively. This gives a contradiction to part (a).
\end{enumerate}  
\end{proof}

\medskip

We now find possible values of the angles which give subsums of $S_3$ from Theorem~\ref{de2} (iii), which we rewrite as:
\begin{equation}\label{(iii)''}
S_3=\sum_{j=1}^{2} \cos\!\left(\varphi+\frac{2j\pi}{3}- \pi\right)+\sum_{k=1}^{4} \cos\!\left(\varphi+\frac{2k\pi}{5}\right),
\end{equation}
where $0<\varphi<\pi$ and $\varphi\neq m\pi/6$, $l\pi/10$ or $s\pi/15$.

\begin{lem}\label{lem-theta-S3}
If $\theta_1+\theta_2+\theta_3=2k\pi$ for $k\in \mathbb Z$ and $\cos(\theta_1)+\cos(\theta_2)+\cos(\theta_3)$
is a subsum of $\pm S_3$ then $\varphi$ is an integer multiple of $\pi/45$. 
\end{lem}

\begin{proof}
Note that all the angles are of the form $\pm(\varphi+j\pi/15)$. 
In particular,
$$
2k\pi=\theta_1+\theta_2+\theta_3= n\varphi+\frac{m\pi}{15},
$$
where $n=\pm1$ or $\pm3$, and $m\in \mathbb{Z}$. This implies that $\varphi$ is an integer multiple of $\pi/45$.
\end{proof}

\medskip

\begin{lem}\label{lem-tech-2}
Suppose that $10<p_3<p_4<\infty$ and $\varphi$ is an integer multiple of $\pi/45$.
If $\cos(2\pi/p_3)+\cos(2\pi/p_4)$ is a subsum of $\pm S_3$. Then 
$$
2\pi/p_3\pm2\pi/p_4=\pi/15
$$
and 
$(p_3,p_4)$ is one of
$$
(15,30),\ (18,45),\ (45,90).
$$
\end{lem}

\begin{proof}
If $\varphi$ is an integer multiple of $\pi/45$ then it is easy to see that this is also
true for all the angles in $S_3$. In particular, if $\cos(2\pi/p_3)+\cos(2\pi/p_4)$ is a subsum of $\pm S_3$ then 
$p_3$ and $p_4$ are both factors of $90$, namely
$$
p_3,p_4\in\{15,18,30,45,90\}
$$

Using Lemma~\ref{tech lemma 1}~(d) we see that one of $\pm 2\pi/p_3$, $\pm2\pi/p_4$ 
has the form $\varphi+2j\pi/3-\pi$ for $j\in\{1,2\}$ and the other has the form
$\varphi+2k\pi/5$ for $k\in\{1,2,3,4\}$. Therefore
\begin{eqnarray*}
\frac{2\pi}{p_3}\pm \frac{2\pi}{p_4}
& = & \pm \left(\varphi+\frac{2j\pi}{3}-\pi-\varphi-\frac{2k\pi}{5}\right) \\
& = & \pm\frac{(10j-6k-15)\pi}{15}.
\end{eqnarray*}
In particular, it is an odd multiple of $\pi/15$. Moreover, since $j$ is coprime to $3$ and
$k$ is coprime to $5$, the numerator is neither a multiple of $3$ nor a multiple of $5$.
Using Lemma~\ref{tech lemma 1}~(a) we see that $0<2\pi/p_3\pm 2\pi/p_4<2\pi/5$. 
Therefore $2\pi/p_3\pm2\pi/p_4=\pi/15$ as claimed.

An inspection of the possible values of $p_3$ and $p_4$ shows that the only solutions are those given above.
\end{proof}

\medskip

\subsection{Proof of Proposition~\ref{main prop 1}}\label{sec 5.2}
We prove Proposition~\ref{main prop 1} by using Theorem~\ref{de2}. The proof is then reduced to showing that if a subsum of equation~\eqref{eq-5.1} vanishes then it has no solutions. We treat the cases separately in Lemmas~\ref{lem cos x=0}, \ref{lem2}, \ref{lem3}, and \ref{lem4}, according to the lengths of the subsums.

We begin with the case where a subsum of length $1$ vanishes. It is clear that the only possibility is $\cos(\theta_j)=0$ for $j=1, 2, 3$. Without loss of generality, we assume that the vanishing cosine is $\cos (\theta_1)$. This implies that $\cos(\theta_j)\ne 0$ for $j=2, 3$, by the condition $\mathrm{Re}\bigl(\mathrm{tr}(R_1R_2R_3)\bigr)<-1$.

\begin{lem}\label{lem cos x=0}
Suppose that $\cos \theta_1 = 0$. Then equation~\eqref{eq-5.1} has no solutions.
\end{lem}

\begin{proof}
Since $\cos(\theta_1)=0$ we have $\theta_1=\pm \pi/2$. Also, using 
$\theta_1+\theta_2+\theta_3=0$ (mod $2\pi$) we have $\theta_3=\mp\pi/2-\theta_2$ 
(mod $2\pi$), $\cos(\theta_3)=\mp\sin(\theta_2)$, and $\cos(\theta_2)=\mp\sin(\theta_3)$. 
Using \eqref{signature (2,1)} this means
$$
\mathrm{Re}\bigl(\mathrm{tr}(R_1R_2R_3)\bigr)=\cos(\theta_2)\mp\sin(\theta_2)=\cos(\theta_3)\mp\sin(\theta_3)<-1
$$
and so, up to adding multiples of $2\pi$, we can assume $\theta_2$, $\theta_3$ are both in the interval $(\pi/2,\pi)$ or both in the interval $(\pi,3\pi/2)$. In particular, $\cos(\theta_2)<0$ and $\cos(\theta_3)<0$. 

Moreover, in this case, equation~\eqref{eq-5.1} becomes
\begin{equation}\label{eq-lem 5.2}
    \cos(\pi-4\phi_1) + \cos(2\phi_3) + \cos(2\phi_4) + \cos (\theta_2) + \cos (\theta_3) = 0,
\end{equation}
and none of the cosine terms vanish.
If this sum has no vanishing proper subsums, then it has no solutions by Lemma~\ref{tech lemma 1}. Hence, by Theorem~\ref{de2}, it suffices to examine the case $\cos a+\cos b=0$.
\begin{enumerate}
    \item[(1)] $a, b\in T_1 \cup T_2$ or $a, b\in T_3$. By Lemma~\ref{lem-subsum-length-2}, there are no solutions in this case.
\medskip
    \item[(2)] $a\in T_1$ and $b\in T_3$, without loss of generality assume $b=\theta_2$. 
    Using $\phi_1=\pi/p_1\le \pi/10$, we have
    $$
    \cos(\theta_2)=-\cos(\pi-4\pi/p_1)=\cos(4\pi/p_1)\ge \cos(2\pi/5)>0.
    $$
    This contradicts the assertion above that $\cos(\theta_2)<0$.

\medskip
    \item[(3)] $a\in T_2$ and $b\in T_3$. Without loss of generality, we assume that $a=2\phi_3$ and $b=\theta_2$. In this case, we have the equations
$$
\qquad\quad
\cos(2\phi_3)+\cos(\theta_2)=0\quad \text{and}\quad
\cos(\pi-4\phi_1) + \cos(2\phi_4) + \cos (\theta_3) = 0.
$$
For the latter equation, by applying sum (i) of Theorem~\ref{de2} 
and Lemma~\ref{tech lemma 1} we see that $2\phi_4=(\pi-4\phi_1)-2\pi/3$. This means
that, writing $\phi_1=\pi/p_1$ and $\phi_4=\pi/p_4$ we have
$$
\theta_3=(\pi-4\pi/p_1)+2\pi/3=5\pi/3-4\pi/p_1 \in(\pi,3\pi/2).
$$
In turn, this means that 
$$
\theta_2=5\pi/2-\theta_3= 5\pi/6+4\pi/p_1 \in(\pi, 3\pi/2).
$$
Using Lemma~\ref{tech lemma 1} we see that we only need to consider
$\phi_1=\pi/p_1$ where $p_1=10$, $11$, $13$, $14$ or $16$. Moreover, using
Proposition~\ref{prop-lp} we can discard the cases $p_1=10$ or $11$ since in these
cases $R_1R_2R_3$ is loxodromic. 
$$
\begin{array}{| l |  l | l | l | l | l |}
\hline
p_1 & 2\phi_1 & 2\phi_4 
& \theta_3 
& \theta_2 
& 2\phi_3 
\\
\hline
13 & 2\pi/13 & \pi/39 & 53\pi/39 & 89\pi/78 & 11\pi/78 \\
14 & \pi/7 & \pi/21 & 29\pi/21 & 47\pi/42 & 5\pi/42 \\
16 & \pi/8 & \pi/12 & 17\pi/12 & 13\pi/12 & \pi/12 \\
\hline
\end{array}
$$
We see that the only case where $\phi_3=\pi/p_3$ with $p_3\in{\mathbb Z}$ is
when $p_1=16$. But in this case, we have $p_3=p_4$ which is a contradiction. 
\end{enumerate}
\end{proof}

\medskip

Next, we consider the case where a subsum of length $2$ vanishes in equation~\eqref{eq-5.1}.

   \begin{lem}\label{lem2}
    Suppose that none of the cosines in equation~\eqref{eq-5.1} vanish, and that there exists a subsum of length $2$ equal to $0$. Then there are no solutions.
\end{lem}
  \begin{proof}
  Recall that equation~\eqref{eq-5.1} is
\begin{center}
$\cos(\pi-4\phi_1)+\cos(2\phi_3)+\cos(2\phi_4)+\cos \theta_1+\cos \theta_2+\cos \theta_3=0$.
\end{center}
By the assumptions, we analyze the cases in which $\cos a+\cos b=0$.
\begin{enumerate}
    \item[(1)] $a, b\in T_1\cup T_2$, or $a, b\in T_3$. In this case, there are no solutions by Lemma~\ref{lem-subsum-length-2}.
\medskip
    \item[(2)] $a\in T_1$ and $b\in T_3$. Without loss of generality, assume that $b=\theta_1$. This implies that 
    $$
    \qquad\quad
    \cos\theta_1=\cos(4\phi_1).
    $$
Then equation~\eqref{eq-5.1} becomes
$$
\qquad\quad
\cos(2\phi_3)+\cos(2\phi_4)+\cos \theta_2+\cos \theta_3=0.
$$
Since there is no vanishing minimal subsum of length $4$, if this equation has solutions, then, by Theorem~\ref{de2}, it must contain a vanishing subsum of length $2$ of the form $\cos c+\cos d=0$.

By Lemma~\ref{lem-subsum-length-2}, the only possible case is that $c\in T_2$ and $d\in T_3$, and we may assume that $c=2\phi_3$ and $d=\theta_2$. It follows that
$$
\qquad\quad
\cos\theta_2=-\cos(2\phi_3)\quad\text{and}\quad
\cos\theta_3=-\cos(2\phi_4).
$$
Thus, we have
$$
\qquad\quad
\theta_2=(\pi\pm 2\phi_3)+2k\pi\quad\text{and}\quad
\theta_3=(\pi\pm 2\phi_4)+2s\pi,
$$
where $k,s\in \mathbb Z$.
Hence, under the assumption that $10\le p_1<p_3<p_4<\infty$, we obtain
$$
\qquad\quad
\cos(\theta_2+\theta_3)=\cos(2\phi_3\pm2\phi_4)>\cos(2\phi_3+2\phi_4)>\cos(4\phi_1)=\cos\theta_1.
$$
This contradicts the condition $\theta_1+\theta_2+\theta_3=2k\pi$ for $k\in \mathbb Z$. Thus, there are no solutions.
\medskip
    \item[(3)] $a\in T_2$ and $b\in T_3$. In this case, by equation~\eqref{eq-5.1}, we obtain a vanishing subsum of length $4$, one of whose terms is $\cos(\pi-4\phi_1)$. Arguing as in (2), this necessarily implies that there exists a subsum of length $2$ containing $\cos(\pi-4\phi_1)$ that is equal to $0$. Hence, we are reduced to either case~(1) or~(2).
\end{enumerate}
\end{proof}

  \medskip

Third, we deal with the case where there exists a vanishing subsum of length $3$ in equation~\eqref{eq-5.1}. Since this equation contains 
six cosines, it means the sum of
the other three must vanish as well.

\begin{lem}\label{lem3}
    Suppose that no subsums of length $1$ or $2$ in equation~\eqref{eq-5.1} vanish, and that there exists a subsum of length $3$ equal to $0$. Then there are no solutions.
\end{lem}

\begin{proof}
Recall that equation~\eqref{eq-5.1} is
\begin{center}
$\cos(\pi-4\phi_1)+\cos(2\phi_3)+\cos(2\phi_4)+\cos \theta_1+\cos \theta_2+\cos \theta_3=0$.
\end{center}
Using Lemma~\ref{lem-subsum-length-2}, we need only consider the case in which $\cos a+\cos b+\cos c=0$, where $a\in T_1$, $b\in T_2$, and $c\in T_3$. Without loss of generality, assume that $b=2\phi_3$ and $c=\theta_1$. Then, we obtain
$$
\cos(\pi-4\phi_1)+\cos(2\phi_3)+\cos (\theta_1)=0
\quad\text{and}\quad
\cos(2\phi_4)+\cos (\theta_2)+\cos (\theta_3)=0.
$$

Solving the second equation using sum~$\mathrm{(i)}$ of Theorem~\ref{de2}, we get
$$
\theta_2+\theta_3=\pm2\pi/3+2s\pi\quad\text{or} \quad\pm4\pi/p_4+2s\pi, \qquad \text{for } s\in\mathbb Z.
$$
Furthermore, using $\theta_1+\theta_2+\theta_3=2k\pi$ for $k\in \mathbb Z$, we obtain 
$$
\cos\theta_1=\cos(\theta_2+\theta_3)=-1/2 \quad \text{or}\quad
\cos(4\pi/p_4).
$$
If $\cos(\theta_1)=-1/2$, then we obtain the equation
$$
\cos(2\pi/p_3)-\cos(4\pi/p_1)=1/2.
$$
Under the condition $p_3>p_1$, this equation admits no solutions by Theorem~\ref{de2}. If $\cos(\theta_1)=\cos(4\pi/p_4)$, then
\begin{eqnarray*}
    \cos(4\pi/p_4) &>& \cos(4\pi/p_1)\ \, \ge \ \, \cos(4\pi/p_1)-\cos(2\pi/p_3).
\end{eqnarray*}
Hence, there are no solutions to the first equation above.
\end{proof}

\medskip

Finally, we treat the case where there is no proper vanishing subsum in equation~\ref{eq-5.1}.
       
\begin{lem}\label{lem4}
Suppose that equation~\eqref{eq-5.1} admits no vanishing proper subsums. Then there are no solutions that satisfy equation~\eqref{eq-5.1}.
\end{lem}

\begin{proof}
The first case is that equation~\eqref{eq-5.1} takes the form of $S_3$, the equation 
from Theorem~\ref{de2}~(iii). We view the angles as either belonging to $S_1$ or $S_2$. 

Since $\cos(\theta_1)+\cos(\theta_2)+\cos(\theta_3)$ is a proper subsum, 
using Lemma~\ref{lem-theta-S3} we see that $\varphi$ must be a multiple of $\pi/45$. 
This means that $p_1$ must be a factor of $180$, while $p_3$ and $p_4$ must be factors of $90$.

It is clear that $\pm (\pi-4\phi_1)$ must be one of the angles from $S_1$ or one from $S_2$.
From Lemma~\ref{tech lemma 1}~(d) we see that one of $\pm2\phi_3$ and $\pm2\phi_4$ must
be one of the angles from $S_1$ and the other from $S_2$. Thus for one choice of 
$j=3,4$ then $\pm(\pi-4\phi_1)$ and $\pm2\phi_j$ are both angles in $S_1$ or both in $S_2$.
Then we can find possible pairs of angles $(p_1,p_j)$ from Lemma~\ref{tech lemma 1}~(b) and
(c). We can also find possible pairs $(p_3,p_4)$ from Lemma~\ref{lem-tech-2}. Using the
fact that $p_1$, $p_3$ and $p_4$ must be factors of $90$ we see the only possibilities 
for $(p_1,p_3,p_4)$ are
$$
(12,15,30),\qquad(18,45,90).
$$
\begin{enumerate}
    \item[(a)] In the case where $(p_1,p_3,p_4)=(12,15,30)$ we have (mod $2\pi$) 
    $$
    \varphi+\frac{2(k+2)\pi}{5}=\pi-\frac{4\pi}{12},\quad 
    \varphi+\frac{2k\pi}{5}=-\frac{2\pi}{15},\quad \phi+\frac{2j\pi}{3}-\pi=-\frac{2\pi}{30}.
    $$
    This translates to $1+3k$ is congruent to $8+5j$ (mod 15). Solving (mod 3) and (mod 5) 
    we obtain $j=1$ and $k=4$ (so $k+2\equiv1$ mod 5). Therefore $\phi=4\pi/15$. 
    The other three angles $\pm\theta_1,\,\pm\theta_2,\,\pm\theta_3$ are (mod $2\pi$)
    $$
    \frac{4\pi}{15}+\frac{4\pi}{3}-\pi=\frac{3\pi}{5},\quad
    \frac{4\pi}{15}+\frac{4\pi}{5}=\frac{16\pi}{15}, \quad
    \frac{4\pi}{15}+\frac{6\pi}{5}=\frac{22\pi}{15}. 
    $$
    There is no choice of signs for which this sums to a multiple of $2\pi$.
    \item[(b)] In the case where $(p_1,p_3,p_4)=(18,45,90)$ we have (mod $2\pi$)
    $$
    \varphi+\frac{2(k+2)\pi}{5}=\pi-\frac{4\pi}{18},\quad 
    \varphi+\frac{2k\pi}{5}=-\frac{2\pi}{90},\quad \phi+\frac{2j\pi}{3}-\pi=\frac{2\pi}{45}.
    $$
    Again, this translates to $1+3k$ is congruent to $8+5j$ (mod $15$) and so 
    $j=1$ and $k=4$. 
    Therefore $\phi=17\pi/45$. The other three angles
    $\pm\theta_1,\,\pm\theta_2,\,\pm\theta_3$ are (mod $2\pi$)
    $$
    \frac{17\pi}{45}+\frac{4\pi}{3}-\pi=\frac{32\pi}{45},\quad
    \frac{17\pi}{45}+\frac{4\pi}{5}=\frac{53\pi}{45}, \quad
    \frac{17\pi}{45}+\frac{6\pi}{5}=\frac{71\pi}{45}. 
    $$
    Again, there is no choice of signs that makes these sum to a multiple of $2\pi$.
\end{enumerate}

Next, noting that $1/2=\cos(\pi/3)$, we see that equation~\eqref{eq-5.1} cannot take the forms (xiv), (xv), or (xvii) in Theorem~\ref{de2}, since none of these forms contains a subsum of the form $\pm\bigl(\cos(2\pi/p_3)+\cos(2\pi/p_4)\bigr)$.

Finally, if equation~\eqref{eq-5.1} takes the form (xiii) in Theorem~\ref{de2}, we obtain
\[
(p_1,p_3,p_4)=(12,14,21)
\qquad\text{with}\qquad
(\theta_1,\theta_2,\theta_3)
=(\pm20\pi/21,\pm16\pi/21,\pm8\pi/21).
\]
It is clear that there is no choice of the $\theta_j$ such that their sum is an integer multiple of $2\pi$. Hence, no solutions exist in this case.
\end{proof}

\medskip

\section{The case $p_1 < p_2$}
\label{section p1<p2}

When $p_1<p_2$, the equation under consideration is
\begin{equation}\label{eq-6.1}
\cos(2\phi_1+2\phi_2)+\cos(2\phi_1-2\phi_2)-\cos(2\phi_3)-\cos(2\phi_4)-\cos \theta_1 -\cos \theta_2-\cos \theta_3=1,    
\end{equation}
where $\phi_j=2\pi/p_j$ for $j=1,2,3,4$.

We exclude the case $p_2=\infty$. In this case, equation~\eqref{eq-6.1} becomes 
\begin{center}
    $2\cos(2\phi_1) - \cos \theta_1 - \cos \theta_2 - \cos \theta_3 = 3$.
\end{center}
According to Theorem~\ref{de2}, the left-hand side of this equation must have six proper subsums, each equal to $1/2$. However, it can be decomposed into at most five subsums. Therefore, this equation admits no solutions.

Then, based on Theorem~3.3 of \cite{sun} and Lemma~\ref{lem p<=9}, we may further assume that $9\le p_1<p_2\le p_3\le p_4<\infty$. 
The main result is as follows:

\begin{prop}\label{prop sec 6}
    Suppose that $9 \leq p_1<p_2 \le p_3\le p_4<\infty$. Then a complex hyperbolic $(p_1, p_2, p_3; p_4)$-triangle group is non-discrete if the element $R_1R_2R_3$ is elliptic.
\end{prop}

By Corollary~\ref{de1}, no minimal sum of length $7$ and value $1$ admits a solution. Therefore, if equation~\eqref{eq-6.1} admits a solution, its terms must decompose into two minimal subsums, each having value $1/2$, together with any vanishing subsums that may be present; all such subsums are listed in Theorem~\ref{de2}.

The proof of Proposition~\ref{prop sec 6} follows by combining Lemmas~\ref{lem-both} and~\ref{lem-only-one} with the discussion following the latter.
In summary, we obtain the main result by analysing those parts of Theorem~\ref{de2} that contain cosine terms of the form $-\cos(2\pi/p)$. 

In what follows we consider the terms $-\cos(2\pi/p_3)$ and $-\cos(2\pi/p_4)$ from 
equation \eqref{eq-6.1}. 

\begin{lem}\label{lem-both}
Suppose that $10\le p_3\le p_4$ and equation~\eqref{eq-6.1} has a proper, minimal subsum 
with value $1/2$ and with $-\cos(2\pi/p_3)-\cos(2\pi/p_4)$ as a subsum. Then 
equation~\eqref{eq-6.1} has no solutions. 
\end{lem}

\begin{proof}
We do this by inspection on Theorem~\ref{de2} (iv) to (xxvi). We see that the only equations 
with a subsum of this form are (xix) with $(p_2,p_4)=(11,66)$, and (xxv) with 
$(p_3,p_4)=(7,42)$ and (xxv) with $(p_3,p_4)=(7,70)$. Since we assumed $p_3\ge 10$ then 
$p_3=11$ and $p_4=66$ and the subsum is Theorem~\ref{de2} (xix). 
Furthermore, since (xix) has length six, the remaining term must be $\cos(\pi/3)=1/2$.

Since $9\le p_1<p_2\le p_3=11$ we must have $(p_1,p_2)=(9,10)$, $(9,11)$ or $(10,11)$.
Therefore $2\pi/p_1-2\pi/p_2$ is one of $\pi/45$, $4\pi/99$ or $\pi/55$. Since none of these angles occurs in (xix) nor is equal to $\pi/3$ we obtain no solutions.
\end{proof}

\medskip

\begin{lem}\label{lem-only-one}
Suppose that $10\le p_3\le p_4$ and equation~\eqref{eq-6.1} has a proper, minimal subsum with value
$1/2$ and which contains only one of $-\cos(2\pi/p_3)$ or $-\cos(2\pi/p_4)$. Then this subsum 
has length three or four and contains $-\cos(2\pi/30)$.    
\end{lem}

\begin{proof}
We use inspection on Theorem~\ref{de2} (iv) to (xxvi). First, the only sums of length one or
two with value $1/2$ do not contain any terms of the form $-\cos(2\pi/p)$ with $p\ge 10$.
Therefore the sum in question must have length at least three. Since our sum contains
only one of the terms $-\cos(2\pi/p_3)$ or $-\cos(2\pi/p_4)$, the other one must be in a
subsum of equation~\eqref{eq-6.1} with value $1/2$, which needs to have length at least three.
Since equation~\eqref{eq-6.1} contains seven terms, our original subsum can have length
at most four.

Now, by inspection, the only sums in Theorem~\ref{de2} with value $1/2$, length three or four and which contain a term of the form $-\cos(2\pi/p)$ with $p\ge 10$ are (vi) and (viii). 
In both cases $p=30$. 
\end{proof}

\medskip

It follows immediately from Lemma~\ref{lem-only-one} that $p_3=p_4=30$. There are two cases to consider. First, suppose that equation~\eqref{eq-6.1} contains two copies of the minimal subsum in part~(vi) of Theorem~\ref{de2}. This necessarily implies that $\cos\theta_j=0$ for some $j$. 
Second, equation~\eqref{eq-6.1} is the sum of Theorem~\ref{de2} (vi) and (viii).
That is:
\begin{equation*}\label{6.1-30-30}
    1=\cos\left(\frac{\pi}{5}\right)+\cos\left(\frac{2\pi}{15}\right)
    +2\cos\left(\frac{4\pi}{15}\right) -\cos\left(\frac{7\pi}{15}\right)
    -2\cos\left(\frac{\pi}{15}\right).
\end{equation*}
In either case, we need to identify which of these terms correspond to 
$\cos(2\pi/p_1-2\pi/p_2)$ and $\cos(2\pi/p_1+2\pi/p_2)$. 
We need to find possible values of $p_1$ and $p_2$ with $9\le p_1<p_2\le 30$ so that 
$2\pi/p_1\pm 2\pi/p_2$ is one of 
$$
\pi/5,\quad 2\pi/15,\quad 4\pi/15,\quad 8\pi/15.
$$
Since $9\le p_1<p_2$ we immediately have
$$
0<2\pi/p_1-2\pi/p_2<2\pi/p_1+2\pi/p_2\le 19\pi/45.
$$
This immediately rules out the term $8\pi/15$, that is $-\cos(7\pi/15)$.
Therefore, we must solve
$$
\frac{2\pi}{p_1}+\frac{2\pi}{p_2}=\frac{a\pi}{15},\quad 
\frac{2\pi}{p_1}-\frac{2\pi}{p_2}=\frac{b\pi}{15}
$$
for integers $2\le b<a\le 4$. Solving these equations gives
$$
p_1=60/(a+b),\quad p_2=60/(a-b).
$$
The only solution is $(p_1,p_2)=(10,30)$ where $a=4$ and $b=2$. In turn, this means that
the other three terms must be $-\cos(\theta_j)$ for $j=1,2,3$. Therefore, we obtain the 
solution
$$
(p_1,p_2,p_3,p_4)=(10,30,30,30),\quad 
(\theta_1,\theta_2,\theta_3)=\pm(4\pi/5,7\pi/15,11\pi/15).
$$

Finally, we show that $\Gamma(10,30,30;30)$ is not discrete. 
Let $v_{23}$ be the fixed point of $R_2R_3$. Then it is not hard to show that,
since $p_2=p_3=p_4$, the map $R_1R_2R_3$
sends the complex line through $v_{23}$ and $R_1(v_{23})$ to itself, and acts there as a
elliptic map with rotation angle $11\pi/15-7\pi/15=4\pi/15$. Therefore $(R_1R_2R_3)^8$ acts
on this complex line with rotation angle $2\pi/15$. A direct computation gives
\begin{eqnarray*}
\lefteqn{
\mathrm{tr}\bigl((R_1R_2R_3)^8R_2(R_3R_2R_1)^8R_2\bigr) }\\
& = & \left(1-2\cos\left(\frac{2\pi}{15}\right)\right)^2-1 \\
& = & 
2-4\cos\left(\frac{2\pi}{15}\right)
+2\cos\left(\frac{4\pi}{15}\right)\\
& = & 
-0.31592\ldots\in(-1,3).
\end{eqnarray*}
Hence, this element is an elliptic isometry with real trace. Its trace can also be written as
\[
2-4\cos\left(\frac{2\pi}{15}\right)
+2\cos\left(\frac{4\pi}{15}\right)
=
1+2\cos\phi,
\]
that is,
\[
\cos\phi
+2\cos\left(\frac{2\pi}{15}\right)
-\cos\left(\frac{4\pi}{15}\right)
=
\frac{1}{2}.
\]
By Theorem~\ref{de2}, $\phi$ is not a rational multiple of $\pi$. Therefore, $\Gamma(10,30,30;30)$ contains an elliptic 
element of infinite order and is consequently non-discrete.

\section{Tool use}

The only software used was Maple, which was used for some calculations. The only AI tools used were DeepSeek and ChatGPT, 
which were only used for typesetting and grammar.

\end{document}